\documentclass[12pt]{amsart}
\usepackage[T1]{fontenc}
\usepackage{lmodern,microtype,mathtools,amssymb,booktabs,array}
\usepackage{tikz}
\usepackage[hidelinks]{hyperref}
\usepackage[capitalise]{cleveref}
\newtheorem{theorem}{Theorem}
\newtheorem{proposition}[theorem]{Proposition}
\newtheorem{lemma}[theorem]{Lemma}

\theoremstyle{remark}

\newcommand{\E}{\mathbb E}
\newcommand{\R}{\mathbb R}
\newcommand{\B}{\mathbb B}
\newcommand{\norm}[1]{\lVert#1\rVert}
\newcommand{\ip}[2]{\langle#1,#2\rangle}
\DeclareMathOperator{\conv}{conv}
\DeclareMathOperator{\diam}{diam}

\title[Borsuk number in four dimensions]{Borsuk number in four dimensions is at most seven}
\author{A. Arman}
\address{Department of Mathematics, University of Manitoba, Winnipeg, MB, R3T 2N2, Canada}
\email{andrew0arman@gmail.com}

\author{A.\ Bondarenko}
\address{Department of Mathematical Sciences, Norwegian University of Science and
	Technology, NO-7491 Trondheim, Norway}
\email{andriybond@gmail.com}
\thanks{The second author was supported in part by Grant 334466 of the Research Council of Norway.}

\author{A.\ Prymak}
\address{Department of Mathematics, University of Manitoba, Winnipeg, MB, R3T 2N2, Canada}
\email{prymak@gmail.com}
\thanks{The third author was supported by NSERC of Canada Discovery Grant RGPIN-2026-06488.}

\author{D.\ Radchenko}
\address{Institut des Hautes \'Etudes Scientifiques, CNRS, Laboratoire Alexander Grothendieck, 35 route de Chartres, Bures-sur-Yvette 91440, France}
\email{danradchenko@gmail.com}
\thanks{The fourth author was supported by ERC Starting Grant No. 101078782.}
\subjclass[2020]{Primary 52C17; Secondary 52A20, 65G30}
\keywords{Borsuk's problem, Gr\"unbaum's covering problem, covering by balls,
	Jung's theorem, universal covering system, computer-assisted proof}
\hypersetup{pdftitle={Borsuk number in four dimensions is at most seven},
	pdfauthor={A. Arman, A. Bondarenko, A. Prymak, D. Radchenko}}
\begin{document}
	\begin{abstract}
		We prove that every bounded set of diameter one in four-dimensional
		Euclidean space can be covered by seven balls of diameter strictly
		smaller than $0.9999$. In particular, the Borsuk number satisfies
		$b_4\le7$. The proof combines geometric arguments based on contact
		points of the circumball with an exact finite computation. For small
		and large circumradius, the covering centres are constructed analytically. In the remaining
		range, a suitable triangle of contact points reduces the problem to
		a four-parameter family of ball polytopes. The parameter cube is
		subdivided into 23\,494 boxes. For each parameter box, candidate centres
		are found by nonlinear optimization, and the resulting covering is
		verified using exact rational interval arithmetic and paraboloid lifting.
	\end{abstract}
	
	\maketitle
	
	\section{Introduction}
	
	Let $\E^n$ be the $n$-dimensional Euclidean space. The \emph{Borsuk
		number} $b_n$ is the smallest integer such that every bounded subset
	of $\E^n$ of positive diameter can be partitioned into at most $b_n$
	sets of strictly smaller diameter. Equivalently, one may use coverings
	instead of partitions. Borsuk~\cite{Borsuk} asked whether $b_n=n+1$
	for every $n$. The vertices of a regular simplex give $b_n\ge n+1$,
	and equality holds for $n\le3$. Kahn and Kalai~\cite{KK} disproved
	the conjecture in sufficiently high dimensions.
	
	Dimension four is the first dimension in which Borsuk's question
	remains open. Lassak's bound $b_n\le2^{n-1}+1$~\cite{Lassak} gives
	$b_4\le9$. Recently, Tolmachev and Voronov~\cite{TV} proved $b_4\le8$
	by a computer-assisted construction: every set of diameter one in
	$\E^4$ can be partitioned into eight sets of diameter less than
	$0.9999$. We reduce the number of parts to seven and obtain a
	covering by balls of diameter less than the same constant.
	
	For a bounded set $A\subset\E^n$ of diameter one, let $N(A)$ be the
	smallest number of Euclidean balls of diameter strictly smaller than
	one whose union contains $A$, and put
	\[
	g_n=\sup_{\diam A=1}N(A).
	\]
	This is the strict-diameter version of Gr\"unbaum's ball-covering
	number~\cite{Grunbaum}. Evidently, $b_n\le g_n$. Our main result is
	the following.
	
	\begin{theorem}\label{thm:main}
		Every bounded set $K\subset\E^4$ of diameter one can be covered by
		seven balls of a common diameter strictly smaller than $0.9999$.
		In particular, $b_4\le g_4\le7$.
	\end{theorem}
	
	For various estimates on $b_n$ and $g_n$, see
	\cite{ABP,ABPR-Grunbaum,Bondarenko,BL,AI-Borsuk,AI-Borsuk-improved,Grinsztajn,JB,OpenAI-Borsuk,Raigorodskii,Schramm}. For a broader account of Borsuk's problem and related questions,
	we refer to Kalai's survey~\cite{Kalai}.

	\subsection{Outline of the proof}
	We first replace $K$ by its compact closure and translate the centre
	of its \emph{circumball}, the smallest closed ball containing it, to
	the origin. Let $r$ denote the squared circumradius. Jung's
	theorem~\cite{Jung} gives $r\le\tfrac25$. The proof uses three ranges:
	\[
	\begin{array}{c@{\qquad}l}
		0\le r\le\tfrac{299}{1000}
		& \text{a cover of the whole circumball},\\[6pt]
		\tfrac{299}{1000}\le r\le\tfrac{38401}{100000}
		& \text{a finite covering computation},\\[6pt]
		\tfrac{38401}{100000}\le r\le\tfrac25
		& \text{a cover determined by five contact points}.
	\end{array}
	\]

	For small circumradius, an explicit arrangement of seven centres
	covers the whole circumball.

	The information supplied by the circumball is very useful for the
	covering problem. The centre of the circumball is a positive convex
	combination of at most five points of $K$ on its boundary, which we
	subsequently refer to as the \emph{contact points}. We write $B(c,R)$
	and $\overline B(c,R)$ for the open and closed balls with centre $c$
	and radius $R$, respectively. Each contact point $p$ gives the
	additional restriction $K\subset\overline B(p,1)$.

	We now consider the intermediate circumradius range. By successively
	choosing a facet nearest to the circumcentre, we find three contact
	points $P_1,P_2,P_3$ forming an acute triangle whose squared
	circumradius $\tau$ satisfies
	\[
	\tfrac56 r\le\tau\le\min\{r,\tfrac13\}.
	\]
	Keeping the circumball and these three distance restrictions gives
	\[
	K\subset H(r,P_1,P_2,P_3):=
	\overline B(0,\sqrt r)\cap\bigcap_{i=1}^3\overline B(P_i,1).
	\]
	These outer sets depend, up to congruence, on four parameters.
	Every set needed in the intermediate range lies in a member of this
	family. This is a universal covering system, in the terminology used
	in~\cite{TV}, see also~\cite{ABPR} for the relation between universal
	covers and Borsuk's problem.

	For each parameter box, we seek seven rational centres such that
	the same seven balls cover every outer set generated by parameters
	in that box. Rational interval bounds give a common outer
	approximation containing all these sets. The substitution
	$q=\norm{x}^2$ turns the conditions that a point belongs to this
	approximation and is missed by the seven open balls into twelve
	linear inequalities in $(x,q)\in\E^5$. It is sufficient to show that
	the resulting polytope lies strictly above the paraboloid
	$q=\norm{x}^2$. Convexity reduces this to a check at its vertices,
	performed with exact rational arithmetic.

	Starting with the whole parameter cube, we try proposed centre
	configurations on each box. If none of the tested configurations satisfies the desired covering condition, we
	bisect one of the box's coordinate intervals and repeat the procedure on
	both smaller boxes. Thus subdivision continues where the covering
	test has not yet succeeded. The resulting subdivision consists of
	$23\,494$ boxes, with a verified covering on every box.
	

	For large circumradius, no subdivision is needed. We choose a
	hyperplane through the origin minimizing the weighted sum of squared
	distances from the five contact points. Scalar projection estimates
	then give a covering away from two cones of half-angle $\pi/4$ about
	the normal by five balls centred at $\tfrac9{20}p_1,\ldots,\tfrac9{20}p_5$,
	where $p_1,\ldots,p_5$ are the contact points. Thus each centre is
	obtained by moving a contact point towards the origin by the same
	factor. Two further balls cover the portions of the cones inside
	the circumball.

	Section~\ref{sec:contacts} proves the reduction to three contact
	points. Section~\ref{sec:analytic} treats the two analytic ranges.
	Sections~\ref{sec:parameters} and~\ref{sec:condition} describe the
	parameter family and a sufficient condition for its coverage.
	Section~\ref{sec:finite} states the finite verification, completes
	the proof, and explains the numerical search used to obtain the
	certificate. The exact interval rules, verification instructions,
	and computational data are supplied separately at
	\begin{center}
		\url{https://github.com/andriyprm/4d7balls}.
	\end{center}
	\noindent
	The exact interval rules in \texttt{VERIFICATION.md} form part of
	the computer-assisted proof. The supplement also gives instructions
	for reproducing the computation from the supplied data.
	
	\subsection{Notation}
	We write $\ip{\cdot}{\cdot}$ and $\norm{\cdot}$ for the Euclidean
	inner product and norm. We use $\B^n=\overline B(0,1)$ for the closed
	unit ball in $\E^n$. The symbols $\diam A$ and
	$\conv A$ denote the diameter and convex hull of $A$.
	A simplex is the convex hull of affinely independent points. A facet
	is a face of codimension one. The relative interior is the interior
	in the affine span. The circumcentre of a simplex is the unique
	point in its affine span equidistant from its vertices, and their
	common distance is its circumradius. Barycentric coordinates are
	the coefficients, with sum one, in the affine representation of
	a point by the vertices. A universal covering system for a class
	of sets is a family such that every set in the class lies in an
	isometric copy of a member of the family.
	
	Throughout the proof, we use the constants
	\[
		\rho=\tfrac{9999}{20000},\qquad \sigma=\rho^2,
		\qquad r_0=\tfrac{38401}{100000}.
	\]
	It suffices to cover each compact set of diameter $1$ by seven open balls of radius
	$\rho$. Indeed, for their centres $c_1,\ldots,c_7$, the continuous
	function $x\mapsto\min_j\norm{x-c_j}$ then has maximum strictly
	smaller than $\rho$. The same centres therefore give a cover by
	closed balls of a common diameter strictly smaller than
	$2\rho=0.9999$.
	
	\section{Contact points of the circumball}\label{sec:contacts}
	
	We begin with the standard contact-point proof of Jung's inequality,
	including the information that will be needed later.
	
	\begin{lemma}\label{lem:contacts}
		Let $K\subset\E^4$ be compact, with $\diam K=1$, and suppose that
		$\overline B(0,\sqrt r)$ is its circumball. There exist
		affinely independent points $p_1,\ldots,p_k\in K$, $k\le5$, and positive
		numbers $\lambda_1,\ldots,\lambda_k$ such that
		\begin{equation}\label{eq:balance}
			\norm{p_i}^2=r,\qquad
			\sum_{i=1}^k\lambda_i=1,\qquad
			\sum_{i=1}^k\lambda_i p_i=0.
		\end{equation}
		Moreover,
		\begin{equation}\label{eq:jung}
			r=\sum_{i<j}\lambda_i\lambda_j\norm{p_i-p_j}^2
			\le\tfrac{1}{2}\left(1-\sum_i\lambda_i^2\right)
			\le\frac{k-1}{2k}\le\tfrac{2}{5}.
		\end{equation}
	\end{lemma}
	\begin{proof}
		The origin belongs to the convex hull of the contact points
		$K\cap\partial B(0,\sqrt r)$. Otherwise strict separation gives a
		direction in which a small movement of the centre decreases the
		distance to every contact point. By compactness, the improvement is uniform
		in a neighbourhood of the contact set, while the complement of that neighbourhood in $K$ has a
		uniform separation from the sphere. A sufficiently small movement would
		therefore decrease the circumradius, a contradiction.
		
		Choose a representation of zero as a convex combination of the contact points
		with the smallest possible support. The weights are positive and the
		points are affinely independent: an affine dependence would allow a
		perturbation of the weights until one became zero. Hence $k\le5$.
		Expanding squared distances gives the equality in~\eqref{eq:jung}.
		The diameter bound gives the first inequality, and
		$\sum_i\lambda_i^2\ge1/k$ gives the second.
	\end{proof}
	
	The next lemma allows us to discard contact points while preserving a large
	circumradius. The nearest facet is useful for two reasons at once: its
	distance from the circumcentre is small, and the perpendicular foot
	lies inside the facet.
	
	\begin{lemma}\label{lem:facet}
		Let $d\ge2$, and let $S$ be a $d$-simplex whose vertices lie on a
		sphere of radius $R$ centred at the origin. Suppose that the origin
		belongs to the relative interior of $S$. Then $S$ has a facet $F$
		such that the circumcentre of $F$ belongs to the relative interior
		of $F$ and the squared circumradius of $F$ is at least
		$(1-d^{-2})R^2$.
	\end{lemma}
	\begin{proof}
		We can assume $S$ is full-dimensional. Write the facet inequalities of $S$
		as $\ip{n_i}{x}\le h_i$, where $n_i$ is an outward unit normal,
		$h_i>0$, and the facet $F_i$ is opposite the vertex $p_i$. Let $\lambda_i>0$
		be the barycentric coordinates of the origin. Since
		$\ip{n_i}{p_j}=h_i$ for $j\ne i$, we have
		\[
		\ip{n_i}{p_i}=-\frac{1-\lambda_i}{\lambda_i}h_i,
		\qquad h_i\le\frac{\lambda_i}{1-\lambda_i}R.
		\]
		Some $\lambda_i$ is at most $1/(d+1)$, so
		$h:=\min_i h_i\le R/d$. Choose $i$ with $h_i=h$. For $j\ne i$,
		\[
		\ip{n_j}{hn_i}=h\ip{n_j}{n_i}<h\le h_j.
		\]
		The strict inequality holds because two distinct facets cannot have
		the same outward unit normal. Since $\ip{n_i}{hn_i}=h_i$, the point
		$hn_i$ lies in the relative interior of $F_i$. The squared distance
		from $hn_i$ to every vertex of $F_i$ is $R^2-h^2$. Therefore $hn_i$
		is the circumcentre of $F_i$, and the squared circumradius of $F_i$
		is at least $(1-d^{-2})R^2$.
	\end{proof}
	
	\begin{proposition}\label{prop:triangle}
		Under the assumptions of Lemma~\ref{lem:contacts}, if $r>\tfrac{1}{4}$, then
		three contact points $P_1,P_2,P_3$ form an acute triangle with squared circumradius $\tau$
		satisfying
		\begin{equation}\label{eq:tau}
			\tfrac{5}{6}r\le\tau\le\min\{r,\tfrac{1}{3}\}.
		\end{equation}
		The circumcentre of the triangle is the orthogonal projection of
		the origin onto the affine plane of the triangle. After an orthogonal
		change of coordinates, these points can be written as
		\[
			P_i=(p_i',h,0)\in\E^2\oplus\E\oplus\E,
			\qquad \norm{p_i'}^2=\tau,\qquad h=\sqrt{r-\tau}.
		\]
		In these coordinates,
		\begin{equation}\label{eq:outer-H}
			K\subset H(r,P_1,P_2,P_3):=
			\overline B(0,\sqrt r)\cap\bigcap_{i=1}^3\overline B(P_i,1).
		\end{equation}
	\end{proposition}
	\begin{proof}
		By~\eqref{eq:jung}, there are at least three contact points in the minimal
		representation. If there are three, their circumcentre is the origin,
		which lies inside the triangle, and $\tau=r$.
		If there are four, apply Lemma~\ref{lem:facet} with $d=3$ to obtain
		a triangular facet with squared circumradius at least $8r/9$.
		If there are five, apply Lemma~\ref{lem:facet} with $d=4$ to obtain
		a tetrahedral facet. Translate the circumcentre of this tetrahedron
		to the origin and apply the same lemma with $d=3$.
		For the triangle obtained in this way we have
		\[
		\tau\ge\tfrac{8}{9}\tfrac{15}{16}r=\tfrac{5}{6}r.
		\]
		In each case the circumcentre belongs to the relative interior of
		the triangle, so the triangle is acute. Applying the
		identity~\eqref{eq:jung} to the three vertices of the triangle,
		after translating the circumcentre to the origin, gives
		$\tau\le\tfrac{1}{3}$.
		
		Let $c$ be the orthogonal projection of the origin onto the affine
		plane of the triangle. Since all three vertices have squared norm
		$r$, the point $c$ is equidistant from these vertices and hence is
		the circumcentre of the triangle. Putting $h=\norm{c}$, Pythagoras
		gives $h^2+\tau=r$, and in particular $\tau\le r$.

		An orthogonal change of coordinates sends the plane through the origin
		parallel to the triangle to the first coordinate plane and $c$ to $(0,0,h,0)$,
		giving the stated coordinates. Finally, $K\subset\overline B(0,\sqrt r)$,
		and the diameter condition gives $\norm{x-P_i}\le1$ for every
		$x\in K$ and $i=1,2,3$. This proves~\eqref{eq:outer-H}.
	\end{proof}
	
	\section{Analytic covers}\label{sec:analytic}
	
	\subsection{Small circumradius}\label{sec:small}
	
	When the circumball is small enough, the contact points are unnecessary.
	The following construction divides the four coordinate directions
	between a regular triangle and two antipodal pairs.
	
	\begin{proposition}\label{prop:small}
		For every $r\ge0$, the ball $\overline B(0,\sqrt r)\subset\E^4$
		can be covered by seven closed balls of radius $\sqrt{5r/6}$.
		If $r\le\tfrac{299}{1000}$, this radius is strictly smaller than $\rho$.
	\end{proposition}
	\begin{proof}
		In an orthogonal decomposition $\E^4=\E^2\oplus\E\oplus\E$,
		take three unit vectors $d_1,d_2,d_3$ forming a regular triangle in
		the first plane, and take $d_4,d_5=\pm e_3$ and
		$d_6,d_7=\pm e_4$. For $x=(u,v,w)$ we have
		\[
		\max_i\ip{x}{d_i}\ge\max\{\norm{u}/2,|v|,|w|\}
		\ge\frac{\norm{x}}{\sqrt6}.
		\]
		Put $c_i=\sqrt{r/6}\,d_i$. If $t=\norm{x}\in[0,\sqrt r]$, then
		\[
		\min_i\norm{x-c_i}^2
		\le t^2+\frac r6-\frac{\sqrt r}{3}t
		\le\frac{5r}{6}.
		\]
	\end{proof}
	
	\subsection{Large circumradius}\label{sec:large}
	
	The five contact points of a circumball with squared radius close
	to $\tfrac{2}{5}$ cannot be concentrated near a proper subspace. We shall express
	this using only their scalar projections. Choose a hyperplane through
	the origin which minimizes the weighted sum of the squared distances
	of the contact points to the hyperplane. Outside two cones about the normal to this hyperplane,
	the projections force one of the five rescaled contact points to be a
	suitable covering centre. The portion of each cone inside the
	circumball is covered by one additional ball, see Figure~\ref{fig:cones}.
	
	The following estimate is the main consequence of the diameter
	condition and bounds the weighted sum of squared projections in
	every direction.
	
	\begin{lemma}\label{lem:projections}
		Suppose that $p_1,\ldots,p_5\in\E^4$ and positive numbers
		$\lambda_1,\ldots,\lambda_5$ satisfy~\eqref{eq:balance}, and that
		$\norm{p_i-p_j}\le1$ for all $i,j$. Put
		\[
		\Phi(v)=\sum_{i=1}^5\lambda_i\ip{v}{p_i}^2,
		\qquad U=\tfrac{1}{2}-r.
		\]
		Then
		\[
			\Phi(v)\le U\norm{v}^2,\qquad v\in\E^4.
		\]
	\end{lemma}
	\begin{proof}
		Fix $v\in\E^4$, and put $a_i=\ip{v}{p_i}$ and
		$w=\sum_i\lambda_i a_i p_i$. Then $\Phi(v)=\ip{v}{w}$, and the
		identity $\sum_i\lambda_i p_i=0$ in~\eqref{eq:balance} gives
		\[
		\sum_i\lambda_i a_i
		=\ip{v}{\sum_i\lambda_i p_i}=0.
		\]
		Since
		$1-\norm{p_i-p_j}^2=1-2r+2\ip{p_i}{p_j}$, expansion gives
		\begin{equation*}
			0\le\sum_{i,j}\lambda_i\lambda_j
			(1-\norm{p_i-p_j}^2)(a_i-a_j)^2
			=4U\Phi(v)-4\norm{w}^2.
		\end{equation*}
		So by Cauchy--Schwarz,
		\[
		\Phi(v)^2\le\norm{v}^2\norm{w}^2
		\le U\norm{v}^2\Phi(v).
		\]
	\end{proof}
	
	For a unit vector $e$, the quantity $\Phi(e)$ is the weighted sum of
	squared distances from the contact points to $e^\perp$. For any
	orthonormal basis $e_1,\ldots,e_4$ of $\E^4$, Pythagoras gives
	\[
	\sum_{j=1}^4\Phi(e_j)
	=\sum_{i=1}^5\lambda_i\sum_{j=1}^4\ip{e_j}{p_i}^2
	=\sum_{i=1}^5\lambda_i\norm{p_i}^2=r.
	\]
	Thus the upper bound on $\Phi$ in two basis directions forces a
	lower bound on the sum of $\Phi$ in the other two directions.
	Choosing a minimizing direction cancels the mixed projections
	and gives a lower bound outside the two cones.
	
	\begin{lemma}\label{lem:cone-projections}
		Under the assumptions of Lemma~\ref{lem:projections}, let $e$ be a
		unit vector minimizing $\Phi$ on the unit sphere, and put
		$\eta=(3r-1)/2$. Then
		\[
			|\ip{x}{e}|\le\frac{\norm{x}}{\sqrt2}
			\quad\Longrightarrow\quad
			\Phi(x)\ge\eta\norm{x}^2.
		\]
	\end{lemma}
	\begin{proof}
		A minimizer exists by continuity and compactness of the unit sphere.
		For any unit $u\perp e$, differentiating
		$\Phi(e\cos t+u\sin t)$ at $t=0$ gives
		\begin{equation}\label{eq:mixed-projections}
			\sum_i\lambda_i\ip{e}{p_i}\ip{u}{p_i}=0.
		\end{equation}
		Complete $e,u$ to an orthonormal basis $e,u,v,w$ of $\E^4$.
		By Pythagoras,
		\[
		\Phi(e)+\Phi(u)+\Phi(v)+\Phi(w)=r.
		\]
		Lemma~\ref{lem:projections} therefore gives
		\[
		\Phi(e)+\Phi(u)\ge r-2U=2\eta.
		\]
		Since $\Phi(u)\ge\Phi(e)$ by minimality, also $\Phi(u)\ge\eta$.
		
		The assertion is immediate for $x=0$. Otherwise put $z=\ip{x}{e}$ and write
		$x=tu+ze$, where $u\perp e$ is a unit vector and
		$t=\norm{x-ze}\ge0$. The hypothesis gives $t^2\ge z^2$.
		Using~\eqref{eq:mixed-projections}, we obtain
		\[
		\begin{split}
			\Phi(x)&=(t^2-z^2)\Phi(u)
			+z^2\bigl(\Phi(u)+\Phi(e)\bigr)\\
			&\ge\eta(t^2-z^2)+2\eta z^2
			=\eta\norm{x}^2.
		\end{split}
		\]
	\end{proof}
	
	For a unit vector $e$ and $R>0$, define the truncated double cone
	\[
	\mathcal C(e,R)=
	\{x\in\E^4:\norm{x}\le R,\ |\ip{x}{e}|\ge\norm{x}/\sqrt2\}.
	\]
	The two cones have axis directions $e$ and $-e$ and half-angle
	$\pi/4$. The next observation explains the two additional centres.
	
	\begin{lemma}\label{lem:cone-cover}
		The set $\mathcal C(e,R)$ is covered by the two closed balls of
		radius $R/\sqrt2$ centred at $\pm(R/\sqrt2)e$.
	\end{lemma}
	\begin{proof}
		For $x\in\mathcal C(e,R)$, put $t=\norm{x}\in[0,R]$. Then
		\[
		\min_\pm\norm{x\mp(R/\sqrt2)e}^2
		=t^2+\frac{R^2}{2}-\sqrt2R|\ip{x}{e}|
		\le t^2-Rt+\frac{R^2}{2}\le\frac{R^2}{2}.
		\]	
	\end{proof}
	
	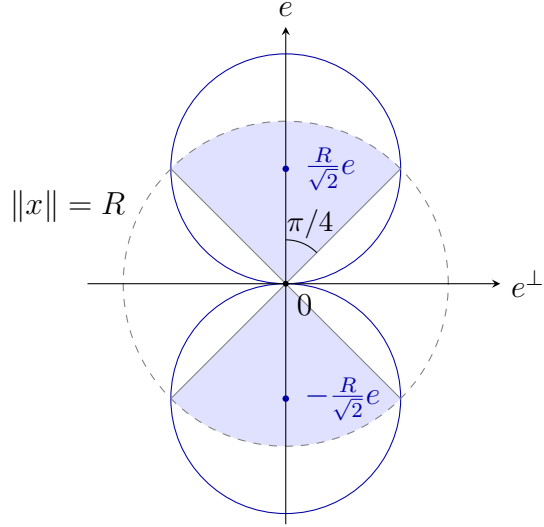
\begin{figure}[tbp]
		\centering
		\begin{tikzpicture}[x=2.15cm,y=2.15cm,>=stealth]
			\fill[blue!12] (0,0)--(45:1) arc[start angle=45,end angle=135,radius=1]--cycle;
			\fill[blue!12] (0,0)--(225:1) arc[start angle=225,end angle=315,radius=1]--cycle;
			\draw[blue!65!black] (0,.707107) circle[radius=.707107];
			\draw[blue!65!black] (0,-.707107) circle[radius=.707107];
			\draw[dashed,gray] (0,0) circle[radius=1];
			\draw[gray] (-.707107,-.707107)--(.707107,.707107);
			\draw[gray] (-.707107,.707107)--(.707107,-.707107);
			\draw[->] (-1.22,0)--(1.32,0) node[right] {$e^\perp$};
			\draw[->] (0,-1.48)--(0,1.58) node[above] {$e$};
			\fill (0,0) circle[radius=.018] node[below right] {$0$};
			\fill[blue!65!black] (0,.707107) circle[radius=.02]
			node[right,xshift=3pt] {$\frac{R}{\sqrt2}e$};
			\fill[blue!65!black] (0,-.707107) circle[radius=.02]
			node[right,xshift=3pt] {$-\frac{R}{\sqrt2}e$};
			\draw (.0,.27) arc[start angle=90,end angle=45,radius=.27];
			\node at (.15,.37) {\small $\pi/4$};
			\node[anchor=east] at (-.92,.48) {$\norm{x}=R$};
		\end{tikzpicture}
		\caption{A two-dimensional section containing the axis $e$.
			The shaded parts of $\overline B(0,R)$ form the section of the truncated double cone.
			They lie in the two balls of Lemma~\ref{lem:cone-cover}.
			In the large circumradius argument, the other directions are covered
			using the five rescaled contact points.}
		\label{fig:cones}
	\end{figure}
	
	We can now combine the projection estimate with the restrictions
	coming from the five contact points. These restrictions give a lower bound
	on every scalar product $\ip{x}{p_i}$. A large weighted sum of their
	squares then forces at least one of them to be positive and sufficiently
	large, bringing $x$ close to the corresponding rescaled contact point.
	
	\begin{proposition}\label{prop:large}
		Let $K\subset\E^4$ be compact of diameter one, with circumball
		$\overline B(0,\sqrt r)$. If $r_0\le r\le\tfrac25$, then $K$ is
		covered by seven open balls of radius $\rho$.
	\end{proposition}
	\begin{proof}
		Since $r_0>\tfrac{3}{8}$, Lemma~\ref{lem:contacts} supplies exactly five
		contact points in the minimal representation. Choose $e$ as in
		Lemma~\ref{lem:cone-projections} and put $\alpha=\tfrac{9}{20}$.
		We use the seven centres
		\[
			\alpha p_1,\ldots,\alpha p_5,
			\qquad \sqrt{r/2}\,e,\qquad-\sqrt{r/2}\,e.
		\]
		We only need to cover
		$H_5:=\overline B(0,\sqrt r)\cap\bigcap_{i=1}^5\overline B(p_i,1)$.
		By Lemma~\ref{lem:cone-cover}, the last two centres cover
		$H_5\cap\mathcal C(e,\sqrt r)$ with squared radius at most
		$r/2\le\tfrac{1}{5}<\sigma$.
		
		Fix $x\in H_5\setminus\mathcal C(e,\sqrt r)$, so that
		$|\ip{x}{e}|<\norm{x}/\sqrt2$, and put
		\[
		q=\norm{x}^2\le r,\qquad d=1-r-q\ge\tfrac{1}{5},
		\qquad T=\max_i\ip{x}{p_i}.
		\]
		The contact restrictions imply $\ip{x}{p_i}\ge-d/2$.
		Taking the weighted sum of the nonnegative quantities
		$(\ip{x}{p_i}+d/2)(T-\ip{x}{p_i})$ and using
		$\sum_i\lambda_i\ip{x}{p_i}=0$, we obtain
		\[
			\Phi(x)\le\frac{dT}{2}.
		\]
		Lemma~\ref{lem:cone-projections} applies to $x$, and hence
		\[
			\min_i\norm{x-\alpha p_i}^2
			=q+\alpha^2r-2\alpha T
			\le L:=q+\alpha^2r-\frac{4\alpha\eta q}{d},
			\qquad\eta=\frac{3r-1}{2}.
		\]
		We now show that $L<\sigma$ throughout the prescribed
		range. Multiplying $L-\sigma$ by $d$ gives
		\[
			(L-\sigma)d=Q(q,r):=-q^2+B(r)q+C(r),
		\]
		where
		\begin{align*}
			B(r)&=1+\sigma+2\alpha-(1+\alpha^2+6\alpha)r,\\
			C(r)&=(\alpha^2r-\sigma)(1-r).
		\end{align*}
		The maximum of $Q(q,r)$ over all real $q$ is
		\[
		\Delta(r)=\frac{B(r)^2}{4}+C(r).
		\]
		Substituting the rational constants gives
		\[
		\Delta(r_0)<0,\qquad
		\Delta'(\tfrac{2}{5})<0,\qquad \Delta''(r)>0.
		\]
		Thus $\Delta'$ is increasing and negative throughout $[r_0,\tfrac{2}{5}]$,
		so $\Delta$ is decreasing and negative there. Since $d>0$, we obtain
		$L<\sigma$. The first five balls therefore cover $H_5\setminus\mathcal C(e,\sqrt r)$.
	\end{proof}
	
	\section{A four-parameter family of outer sets}\label{sec:parameters}
	
	By Proposition~\ref{prop:triangle}, the remaining task is to cover the sets
	in~\eqref{eq:outer-H} for
	\begin{equation}\label{eq:remaining}
		\tfrac{299}{1000}\le r\le r_0,
		\qquad \tfrac{5}{6}r\le\tau\le\min\{r,\tfrac{1}{3}\}.
	\end{equation}
	We describe the whole range by one closed cube $[0,1]^4$.
	The only change at $r=\tfrac{1}{3}$ concerns the lower bound on the
	height of the triangle's plane: this bound is zero for $r\le\tfrac13$
	and $\sqrt{r-\tfrac13}$ for $r\ge\tfrac13$. A signed square in the
	first parameter incorporates this change without requiring a second family.
	
	Put
	\[
	\delta_-:=\tfrac{1}{3}-\tfrac{299}{1000}=\tfrac{103}{3000},
	\qquad \delta_+:=r_0-\tfrac{1}{3}=\tfrac{15203}{300000}.
	\]
	For $u=(u_1,u_2,u_3,u_4)\in[0,1]^4$, define $t_+=\max\{t,0\}$ and
	\begin{align}
		t&=u_1\sqrt{\delta_+}-(1-u_1)\sqrt{\delta_-},
		&r&=\tfrac{1}{3}+t|t|,\nonumber\\
		h&=(1-u_2)\sqrt{r/6}+u_2t_+,
		&\tau&=r-h^2,\label{eq:single-rh}\\
		a_0&=\sqrt{\max\{0,1-1/(4\tau)\}},
		&a&=\frac{1-u_3}{2}+u_3a_0,\nonumber\\
		b&=(1-u_4)\sqrt{(1-a)/2}+u_4a.&&\label{eq:single-ab}
	\end{align}
	The three contact points are
	\begin{align}
		P_1&=(\sqrt\tau,0,h,0),\nonumber\\
		P_2&=(\sqrt\tau(2a^2-1),\,2\sqrt\tau\,a\sqrt{1-a^2},\,h,0),
		\label{eq:single-contacts}\\
		P_3&=(\sqrt\tau(2b^2-1),\,-2\sqrt\tau\,b\sqrt{1-b^2},\,h,0).
		\nonumber
	\end{align}
	All formulas are continuous on the closed cube. The signed square
	is useful for interval estimates: $\sqrt{\max\{0,r-\tfrac{1}{3}\}}=t_+$
	is a Lipschitz function of $u_1$, including where $r$ passes through
	$\tfrac{1}{3}$.
	
	\begin{lemma}\label{lem:single-onto}
		Every outer set in~\eqref{eq:outer-H} satisfying~\eqref{eq:remaining}
		is congruent to one represented by
		\eqref{eq:single-rh}--\eqref{eq:single-contacts}.
	\end{lemma}
	\begin{proof}
		The quantity $t$ increases from $-\sqrt{\delta_-}$ to
		$\sqrt{\delta_+}$, and $t\mapsto\tfrac{1}{3}+t|t|$ is strictly increasing.
		Thus $r$ runs through the whole required interval. For each fixed $r$,
		the height $h$ runs from $\sqrt{r/6}$ to
		$t_+=\sqrt{\max\{0,r-\tfrac{1}{3}\}}$. These endpoints are in the stated
		order because $r\le r_0<\tfrac{2}{5}$. Hence $\tau=r-h^2$ runs through
		$[5r/6,\min\{r,\tfrac{1}{3}\}]$.
		
		Order the angles of the acute triangle as $A\ge B\ge C$.
		Then $A\ge\pi/3$. By the sine rule, the longest-side condition is
		$2\sqrt\tau\sin A\le1$. For $\tau>\tfrac{1}{4}$ this gives
		$\cos A\ge\sqrt{1-1/(4\tau)}$; for $\tau\le\tfrac{1}{4}$ the
		longest-side condition imposes no additional restriction on an acute angle.
		Thus every required value $a=\cos A$ belongs to $[a_0,\tfrac{1}{2}]$.
		For fixed $A$, the inequalities $C\le B\le A$
		are equivalent to $(\pi-A)/2\le B\le A$, or
		\[
		a\le b=\cos B\le\sqrt{(1-a)/2}.
		\]
		These are precisely the intervals in~\eqref{eq:single-ab}.
		Substituting $a=\cos A$ and $b=\cos B$ into~\eqref{eq:single-contacts}
		and using $A+B+C=\pi$ gives
		\[
		\norm{P_1-P_2}=2\sqrt\tau\sin A,\quad
		\norm{P_1-P_3}=2\sqrt\tau\sin B,\quad
		\norm{P_2-P_3}=2\sqrt\tau\sin C.
		\]
		By the sine rule, these are the side lengths of the required triangle.
		The circumcentres of both triangles are the orthogonal projections
		of the origin onto their respective affine planes, at distance $h$
		from the origin. Hence an orthogonal change of coordinates and a
		relabelling identify the corresponding outer sets.
	\end{proof}
	
	The closed cube also represents right triangles and limiting
	configurations in which two contact points coincide.
	This causes no difficulty: coverage is proved for the entire closed
	family, which includes all the sets required by the reduction.
	
	\section{A sufficient covering condition}\label{sec:condition}
	
	We now describe a finite test that certifies one set of seven centres
	for a whole parameter box. The condition is sufficient: failure of the test does not rule out
	a covering, even for the same centres.
	
	\subsection{A common rational outer approximation}\label{sec:outer-approximation}
	
	Let $Q\subset[0,1]^4$ be a closed parameter box. Suppose that rational
	intervals $[r_-,r_+]$ and $[\ell_{ij},u_{ij}]$ enclose the squared
	circumradius and the coordinates of each contact point for all parameters
	in $Q$. Define
	\[
		A_i=\left(\frac{\ell_{ij}+u_{ij}}2\right)_{j=1}^4,
		\qquad e_{ij}=\frac{u_{ij}-\ell_{ij}}2,
		\qquad E_i=\tfrac{2}{3}\sum_{j=1}^4e_{ij}.
	\]
	For any $x$ in a true outer set represented by $Q$, we have
	$|x_j|\le\norm{x}\le\sqrt r\le\sqrt{\tfrac{2}{5}}<\tfrac{2}{3}$. Hence
	$|\ip{P_i-A_i}{x}|\le E_i$. Expanding $\norm{x-P_i}^2\le1$
	and using $\norm{P_i}^2=r$ gives
	\begin{equation}\label{eq:rational-body}
		\norm{x}^2\le r_+,\qquad
		\norm{x}^2-2\ip{A_i}{x}\le b_i:=1-r_-+2E_i,
		\quad 1\le i\le3.
	\end{equation}
	Thus the set $\Omega_Q$ defined by~\eqref{eq:rational-body} contains
	every member of the family represented by $Q$.
	
	This step is what changes a calculation at a sample parameter into a
	uniform statement on a box. The intervals account for all movement of
	the contact points, and the terms $2E_i$ allow for the resulting error in
	the distance inequalities. The supplement gives the exact interval rules. They use outward
	rounding on the grid $2^{-44}\mathbb Z$, with square roots bounded
	by integer square roots. Thus every enclosure is obtained using
	integer arithmetic.
	
	\subsection{Lifting to a polytope}
	
	We use the paraboloid lifting $x\mapsto(x,\norm{x}^2)$, a standard
	construction in computational geometry, particularly in the theory
	of Voronoi diagrams; see, for example, the textbook treatment in
	Matou\v{s}ek~\cite[Section~5.7]{Matousek}. This lifting is useful here
	because membership in a ball, or in the complement of an open ball,
	becomes a linear inequality in the lifted coordinates.
	Edelsbrunner and Seidel~\cite[Sections~2 and~5]{EdelsbrunnerSeidel}
	use this construction for intersection and union problems involving
	balls. Runge~\cite{Runge} gives a related exact test for covering an
	intersection of balls by a union of balls, using quadratic
	optimization over polyhedra in the original dimension.
	Here we combine the lifting with the rational outer approximation
	above and a sufficient covering test at the vertices of a polytope.
	
	Let $C_1,\ldots,C_7\in\E^4$ be rational centres, and let $s$ be
	a rational number with $0<s\le\sigma$. The value of $s$ may vary
	from one box to another. Introduce a new real variable $q$ and
	define a polyhedron (an intersection of finitely many closed
	half-spaces) $\mathcal P_Q\subset\E^5$ by
	\begin{align}
		0\le q&\le r_+,\label{eq:lift-radius}\\
		q-2\ip{A_i}{x}&\le b_i, &&1\le i\le3,\notag\\
		2\ip{C_j}{x}-q&\le\norm{C_j}^2-s, &&1\le j\le7.
		\label{eq:lift-miss}
	\end{align}
	There are twelve inequalities in all. If $x\in\Omega_Q$ is outside of every open ball $B(C_j,\sqrt s)$, then
	$(x,\norm{x}^2)\in\mathcal P_Q$.
	
	\begin{proposition}[Sufficient covering condition]\label{prop:condition}
		Suppose that five of $C_1,\ldots,C_7$ are affinely independent and
		zero is an interior point of their convex hull. If
		$\mathcal P_Q$ is empty, or if every vertex $(x,q)$ of $\mathcal P_Q$
		satisfies
		\begin{equation}\label{eq:negative}
			F(x,q):=\norm{x}^2-q<0,
		\end{equation}
		then the seven open balls $B(C_j,\sqrt s)$ cover $\Omega_Q$.
	\end{proposition}
	\begin{proof}
		Suppose, for a contradiction, that $x_0\in\Omega_Q$ lies outside
		all seven open balls. Then $(x_0,\norm{x_0}^2)\in\mathcal P_Q$,
		so $\mathcal P_Q$ is nonempty.
		
		First we prove that $\mathcal P_Q$ is bounded. A recession direction $(y,t)$, meaning a direction in which one can
		move arbitrarily far from a feasible point while remaining feasible, has $t=0$ by~\eqref{eq:lift-radius}, and
		$\ip{C_j}{y}\le0$ by~\eqref{eq:lift-miss}. Take strictly positive
		barycentric weights for the five selected centres. Their weighted
		sum is zero, so all five inner products must vanish. These centres
		span $\E^4$, and hence $y=0$.
		
		A bounded polyhedron is a polytope, and every nonempty polytope
		is the convex hull of the vertices. Since $F$ is convex, the value of $F$
		at a convex combination of vertices does not exceed the corresponding
		convex combination of the values of $F$ at those vertices. Thus~\eqref{eq:negative} holds throughout $\mathcal P_Q$.
		This contradicts $F(x_0,\norm{x_0}^2)=0$.
	\end{proof}
	
	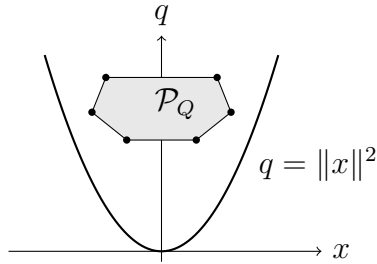
\begin{figure}[ht]
		\centering
		\begin{tikzpicture}[scale=0.92]
			\draw[->] (-2.2,0)--(2.3,0) node[right] {$x$};
			\draw[->] (0,-.15)--(0,3.1) node[above] {$q$};
			\filldraw[fill=gray!18,draw=black] (-1,2)--(-.5,1.6)--(.5,1.6)--(1,2)--(.8,2.5)--(-.8,2.5)--cycle;
			\draw[thick,domain=-1.68:1.68,samples=80] plot (\x,{\x*\x});
			\node at (.2,2.12) {$\mathcal P_Q$};
			\node[anchor=west] at (1.24,1.25) {$q=\norm{x}^2$};
			\foreach \x/\y in {-1/2,-.5/1.6,.5/1.6,1/2,.8/2.5,-.8/2.5}
			\fill (\x,\y) circle (1.5pt);
		\end{tikzpicture}
		\caption{A schematic illustration of the finite covering condition.
			The actual polytope lies in $\E^5$. If every vertex satisfies
			$q>\norm{x}^2$, convexity gives the same inequality throughout the
			polytope, excluding a lifted uncovered point.}
		\label{fig:lift}
	\end{figure}
	
	All parts of this condition are checked with rational arithmetic.
	The positive barycentric weights are obtained by solving a $5\times5$
	linear system for each candidate five-centre simplex. To find the
	vertices, the verifier enumerates every five-element subset of the
	twelve inequalities, solves each resulting nonsingular system,
	and tests all twelve inequalities at the solution.
	Every vertex occurs in this list, including when more than five
	constraints are active (hold with equality) or the polyhedron has smaller dimension.
	Indeed, the active normals at a vertex span $\E^5$; otherwise a
	small movement in both directions orthogonal to their span would
	remain feasible, contradicting extremality.
	
	An empty list of feasible vertices is used only after boundedness
	has been proved. Under this assumption, the absence of feasible
	vertices implies that the polyhedron is empty.
	No numerical tolerance is used for singularity, feasibility, or the
	strict sign in~\eqref{eq:negative}.
	
	\section{The finite certificate and the main theorem}\label{sec:finite}
	
	The preceding reduction leaves one finite statement to be verified.
	A \emph{dyadic box} in $[0,1]^4$ is a product
	\[
	Q=\prod_{j=1}^4
	\left[\frac{k_j}{2^{\ell_j}},\frac{k_j+1}{2^{\ell_j}}\right],
	\qquad 0\le k_j<2^{\ell_j},
	\]
	where $k_j$ and $\ell_j$ are integers and $\ell_j\ge0$.
	The depth of $Q$ is $\ell_1+\cdots+\ell_4$. A subdivision is obtained by
	starting with the whole cube and repeatedly bisecting one coordinate
	interval of a box, retaining both children.
	
	\begin{proposition}[Computer-assisted verification]\label{prop:computational}
		There is a subdivision of $[0,1]^4$ into $23\,494$ dyadic boxes
		with the following property. For each terminal box $Q$, seven
		centres $C_{Q,1},\ldots,C_{Q,7}\in2^{-20}\mathbb Z^4$ and a rational
		number $s_Q$, with $0<s_Q\le\sigma$, satisfy the hypotheses of
		Proposition~\ref{prop:condition} for the rational outer set $\Omega_Q$.
		Consequently, every set represented by the parameter cube is covered
		by seven open balls of radius $\rho$.
	\end{proposition}
	
	\begin{proof}[Verification of Proposition~\ref{prop:computational}]
		The subdivision, centres, and squared radii are recorded in
		\texttt{search\_certificate.json.gz}. The program
		\texttt{seven\_balls\_single.py} checks these data as follows.
		
		First, the program reconstructs the binary subdivision from the whole cube.
		Every split must have both children, every terminal box must have
		the required centres and squared radius, and the input must end
		exactly when all boxes have been accounted for. This establishes
		coverage of the entire cube. Adjacent boxes share their boundary,
		and all interval enclosures are closed.
		
		For each terminal box, the program computes the rational interval
		enclosures using the exact rules given in the supplement and forms
		$\Omega_Q$ as in Section~\ref{sec:outer-approximation}.
		The program checks that five centres form a simplex with the
		origin as an interior point, then enumerates all $\binom{12}{5}=792$
		possible vertex bases of the lifted polytope, solves every
		nonsingular system, and tests feasibility and the strict inequality
		$\norm{x}^2-q<0$ at each feasible solution. These calculations use
		only integer arithmetic, integer square roots, and rational
		Gaussian elimination.
		
		All $23\,494$ terminal boxes pass these checks. The subdivision has
		$23\,493$ splits and maximum depth $29$. The volumes of the terminal
		boxes sum exactly to one, with no unresolved boxes. At $23\,466$ leaves the
		stored squared radius is smaller than $\sigma$; the remaining $28$
		use $\sigma$. Proposition~\ref{prop:condition} gives the required
		covers. The complete verification report and instructions for
		repeating this calculation are included in the supplement.
	\end{proof}
	
	The numerical search is not used by the verifier. In particular,
	the verification does not depend on the sample points, optimization
	logs, initial configurations, or the order in which the subdivision
	was found. The supplied program uses exact arithmetic but is not
	a proof-assistant formalization. The interval rules and finite
	verification in the program implement the mathematical argument above.
	
	\begin{proof}[Proof of Theorem~\ref{thm:main}]
		Replacing $K$ by the closure of $K$, we may assume that $K$ is compact.
		Translate the centre of the circumball of $K$ to the origin and
		let $r$ denote the squared circumradius. Lemma~\ref{lem:contacts} gives $r\le\tfrac25$.
		Proposition~\ref{prop:small} covers $K$ by seven open balls of radius
		$\rho$ when $r\le\tfrac{299}{1000}$, and
		Proposition~\ref{prop:large} does so when $r\ge r_0$.
		
		In the remaining range, $r>\tfrac14$. Proposition~\ref{prop:triangle}
		gives an acute triangle of contact points satisfying~\eqref{eq:tau}.
		The corresponding outer set~\eqref{eq:outer-H} contains $K$ and,
		by Lemma~\ref{lem:single-onto}, is congruent to a member of the
		parameter family. Proposition~\ref{prop:computational} covers this
		member by seven open balls of radius $\rho$.
		
		In all cases, compactness allows the common radius to be decreased
		strictly below $\rho$. The resulting common diameter is strictly
		smaller than $2\rho=0.9999$. Assigning each point to one of the
		balls containing that point also gives the asserted bound on $b_4$.
	\end{proof}
	
	\subsection{Construction of the certificate}\label{sec:finder}
	We describe the search to explain how the finite data were obtained.
	For a parameter box $Q$, the task is to choose centres covering
	the associated common outer approximation $\Omega_Q$. If the search does
	not produce a certified configuration, the box is bisected and both children are kept.
	The exact test of Proposition~\ref{prop:condition} is applied before
	any box is accepted.
	
	The search uses boundary samples of $\Omega_Q$. For a unit vector
	$u$, the radial boundary point is $R(u)u$, where
	\[
		R(u)=\min\left\{\sqrt{r_+},\,
		\min_{1\le i\le3}
		\left(\ip{A_i}{u}+\sqrt{\ip{A_i}{u}^2+b_i}\right)\right\}.
	\]
	Here $b_i>0$, so the origin belongs to $\Omega_Q$, and the formula
	follows by solving the quadratic inequalities along a ray.
	We generate a scrambled Sobol sequence in $[0,1]^4$, apply the inverse
	standard normal distribution function coordinatewise, and normalize
	the resulting vectors to obtain sample directions.
	
	For a finite sample $X$, the numerical objective is
	\[
	f_X(C_1,\ldots,C_7)
	=\max_{x\in X}\min_{1\le j\le7}\norm{x-C_j}^2.
	\]
	All $28$ coordinates of the centres are free. We alternate two
	local procedures: assigning each sample point to a nearest
	centre and approximately minimizing the maximum distance from each
	centre to the sample points assigned to that centre;
	and refining all centres jointly using smooth approximations to
	the minimum and maximum in $f_X$. Poorly covered points from a
	larger sample are added when needed. The supplement specifies
	the smooth objective and the optimization parameters.
	
	Neighbouring parameter boxes often admit the same centres. Before
	starting a new optimization, the search therefore tries configurations
	inherited from the parent box and from previously treated nearby
	parameters. Initial configurations were obtained from selected
	triangles; no symmetry is imposed on the final centres. The
	supplementary seed file records these and further configurations
	from preliminary searches.
	
	After refinement, the centre coordinates are rounded to multiples
	of $2^{-20}$ and submitted to the exact test. Once a configuration
	passes at $s=\sigma$, the centre coordinates are held fixed while numerical
	bisection proposes a smaller squared radius. Every retained value
	$s_Q$ is checked again by the exact test. If no proposal succeeds,
	the already certified value $\sigma$ is retained. Thus numerical
	optimization affects the size of the subdivision and the quality
	of the proposed radii, while coverage is established separately
	by Proposition~\ref{prop:condition}. The supplement gives the
	optimization parameters and commands for a fresh search.
	
	\section*{History of the work and AI use disclosure}
	An earlier version of this work established the weaker bound
	$g_4\le8$ using a universal covering family of polytopes and nonlinear
	optimization to find covering centres. In that version, large language
	models were used only for coding assistance, and the approach required
	substantially more computation.
	
	We subsequently explored the problem with ChatGPT Astra, initially
	seeking purely analytic arguments without computational assistance.
	This led to the approach based on the circumradius and substantially
	reduced the computational search space for coverings of ball polytopes. ChatGPT was also used for coding and editing. We checked
	all mathematical arguments and the verifier code, and made an effort to simplify the exposition and cite references relevant to the methods used.

\end{document}